\documentclass[10pt, twoside, letterpaper]{amsart}

\usepackage{t1enc}
\usepackage{latexsym}
\usepackage{amssymb}
\usepackage{graphicx}
\usepackage{amsmath}
\usepackage{amsthm}
\usepackage{amsfonts}
\usepackage{mathrsfs}
\usepackage{enumitem}
\usepackage[all]{xy}
\usepackage[british]{babel}
\usepackage{color}
\usepackage{xcolor}
\usepackage[hypertexnames=false,
    pdftex,
	pdfpagemode=UseNone,
	breaklinks=true,
	extension=pdf,
	colorlinks=true,
	linkcolor=blue,
	citecolor=blue,
	urlcolor=blue,
]{hyperref}

\newcommand{\cyclic}{\mathop{\kern0.9ex{{+}\kern-2.10ex\raise-0.20
      ex\hbox{\Large\hbox{$\circlearrowright$}}}}\limits}
\newcommand{\acts}{\mbox{\,\raisebox{0.26ex}{\tiny{$\bullet$}}\,}}

\newtheoremstyle{daniel}{3.0mm}{2mm}{\itshape}{}{\bfseries}{.}{1.5mm}{}
\theoremstyle{daniel}
\newtheorem{thm}{Theorem}[section]
\newtheorem{prop}[thm]{Proposition}
\newtheorem{Defi}[thm]{Definition}
\newtheorem{lemma}[thm]{Lemma}
\newtheorem{cor}[thm]{Corollary}
\newtheorem{Exs}[thm]{Examples}
\newtheorem{Rems}[thm]{Remarks}

\newtheorem*{thm*}{Theorem}
\newtheorem*{cor*}{Corollary}
\newtheorem*{thm3.6}{Theorem 3.6}
\newtheorem*{thm4.3}{Theorem 4.3}

\newtheorem*{prop*}{Proposition}
\newtheorem*{Notation}{Notation}

\newtheorem{Prop}[thm]{Proposition}

\newtheorem{Def}[thm]{Definition}

\newtheorem{Rem}[thm]{Remark}

\newtheorem{Ex}[thm]{Example}

\newtheorem*{Setup}{Setup}

\newenvironment{rem}   {\begin{Rem}\em}{\end{Rem}}

\newenvironment{defi}  {\begin{Defi}\em}{\end{Defi}}

\def\cC{\mathcal{C}}
\def\cD{\mathcal D }
\def\cE{\mathcal E}
\def\cF{\mathcal F}
\def\cH{\mathcal H}
\def\cM{\mathcal M }

\def\cL{\mathcal L}
\def\cO{\mathcal O}

\def\cU{\mathcal U}
\def\cW{\mathcal W}
\newcommand{\cX}{\mathcal{X}}
\def\cY{\mathcal Y}
\def\cZ{\mathcal Z}

\def\mg{\mathfrak{m}}

\newcommand{\C}{\mathbb{C}}
\newcommand{\N}{\mathbb{N}}

\newcommand{\R}{\mathbb{R}}

\DeclareMathOperator{\Aut}{Aut}

\DeclareMathOperator{\Coh}{Coh}
\DeclareMathOperator{\coker}{Coker}
\DeclareMathOperator{\id}{id}
\DeclareMathOperator{\im}{Im}
\DeclareMathOperator{\End}{End}
\DeclareMathOperator{\Ext}{Ext}
\DeclareMathOperator{\E}{E^1}
\DeclareMathOperator{\GL}{GL}
\DeclareMathOperator{\Hom}{Hom}
\DeclareMathOperator{\Ker}{Ker}
\DeclareMathOperator{\Pic}{Pic}
\DeclareMathOperator{\Quot}{Quot}
\DeclareMathOperator{\rank}{rank }
\DeclareMathOperator{\Spec}{Spec}

\DeclareMathOperator{\Todd}{Todd}
\DeclareMathOperator{\ch}{ch}
\def\Def{\mbox{Def}}

\def\hq{\hspace{-0.5mm}/\hspace{-0.14cm}/ \hspace{-0.5mm}}

\DeclareMathOperator{\Id}{Id}

\numberwithin{equation}{section}

\begin{document}
\title[Moduli spaces for $\omega$-semistable sheaves]{Moduli spaces of sheaves that are semistable with respect to a K\"ahler polarisation}
\author{Daniel Greb}
\address{Daniel Greb\\Essener Seminar f\"ur Algebraische Geometrie und Arithmetik\\Fakult\"at f\"ur Ma\-the\-matik\\Universit\"at Duisburg--Essen\\
45112 Essen\\ Germany}
\email{daniel.greb@uni-due.de}
\urladdr{\href{http://www.esaga.uni-due.de/daniel.greb/}
{http://www.esaga.uni-due.de/daniel.greb/}}

\author{Matei Toma}
\address{Matei Toma, 
Universit\'e de Lorraine, CNRS, IECL, F-54000 Nancy, France}

\email{Matei.Toma@univ-lorraine.fr}
\urladdr{\href{http://www.iecl.univ-lorraine.fr/~Matei.Toma/}{http://www.iecl.univ-lorraine.fr/~Matei.Toma/}}

\date{\today}


\begin{abstract}
Using an existence criterion for good moduli spaces of Artin stacks by Alper--Fedorchuk--Smyth 
we construct a proper moduli space of rank two sheaves with fixed Chern classes on a given complex projective manifold that are Gieseker-Maruyama-semistable with respect to a fixed K\"ahler class. 
  \end{abstract}
\maketitle
\setcounter{tocdepth}{1}
\noindent

\section{Introduction}

Moduli spaces of sheaves of fixed topological type that are Gieseker-semistable with respect to a given ample class on a projective manifold $X$ have been studied for several decades. When one studies the way these moduli spaces vary if the polarisation changes, examples show that in dimension bigger than two one encounters sheaves $\mathcal{E}$ that are Gieseker-semistable with respect to non-rational, real ample classes $\alpha \in \mathrm{Amp}(X)_\mathbb{R}$ on $X$, i.e., that enjoy the property that for a K\"ahler form $\omega$ representing $\alpha$ and for every proper coherent subsheaf $\mathcal{F} \subset \mathcal{E}$ we have $p_{\mathcal{F}}(m) \leq p_\mathcal{E}(m)$ for all $m$ sufficiently large, where the reduced Hilbert polynomial $p_E(m)$ with respect to $\alpha = [\omega]$ is defined by
\[p_E(m) = \frac{1}{\rank(E)}\int_X ch(E) e^{m\omega} \Todd(X),\]see for example \cite{GRTsurvey}. When $\omega$ represents the first Chern class of an ample line bundle $L$, the Riemann--Roch theorem states that $p_{E}(m)$ equals $\frac{1}{\rank(E)} \chi(E\otimes L^m)$, and so the above generalises the notion of Gieseker--stability from integral classes to real classes, and in fact to all K\"ahler classes $[\omega]$. Both in the case of a real ample polarisation and of an arbitrary K\"ahler class on a compact K\"ahler manifold, the question arises whether there is a moduli space for such sheaves. In fact, it seems that the problem of constructing such moduli spaces was explicitly posed quite some time ago by Tyurin, see the discussion in \cite[Sect.~3.2]{TelemanCommContMath}.

When semistability is measured with respect to an ample line bundle, the construction of moduli spaces is based on Geometric Invariant Theory, and hence of global nature. Using the special structure of cones of positive classes and Geometric Invariant Theory for moduli spaces of quiver representations, it was shown by the authors in joint work \cite{GRT1} with Julius Ross that a GIT-construction of projective moduli spaces for $\omega$-semistable sheaves can still be carried out on projective threefolds.  When dealing with arbitrary compact K\"ahler manifolds it is however quite unlikely that a finite-dimensional, global construction of a moduli space is possible. As an alternative approach, it is natural to study the symmetries induced by automorphism groups on semi-universal deformation spaces and to carry out a functorial local construction from which in the end the moduli space is glued. This approach is most naturally pursued in the language of analytic/algebraic stacks. Using recent advances in this theory, both regarding the correct type of moduli space to construct \cite{Alper13} and regarding existence criteria \cite{AlperFS}, in this paper we establish the following main result:

\begin{thm*} The algebraic stack of $\omega$-semistable sheaves of rank two and given Chern classes admits a good moduli space that is a proper algebraic space; in particular, the moduli space is separated.
\end{thm*}

We emphasise that this in particular yields a new construction of the Gieseker-moduli space in the case where the polarisation is given by an ample line bundle. We do not expect the restriction to the rank two case to be necessary; here, it simplifies the analysis of the local slice models describing the action of the automorphism groups of stable sheaves on their semi-universal deformation space. Note however that the theorem stated above does not claim that the moduli space is projective  or even a scheme; new methods seem to be needed to investigate these additional questions.

While the approach followed here is very promising in the general K\"ahler case, both fundamental work extending \cite{AlperFS} to the analytic setup and a finer analysis of the geometry of the symmetries of semi-universal analytic deformation spaces will be needed to attack the existence question for semistable sheaves on compact K\"ahler manifolds.

\subsection*{Structure of the paper}
In Section~\ref{sect:prelim}, we collect the basic notions and their fundamental properties. More precisely, Section~\ref{subsect:sheafextensions} discusses sheaf extensions and their automorphisms, in Section~\ref{subsect:coherentsheaves} we introduce the notion of Gieseker-semistability with respect to a K\"ahler class and establish the basic properties of this notion, in Section~\ref{subsect:rank2} we provide the structure theory of semistable sheaves of rank two, and in Section~\ref{subsect:quotient_presentations} we establish the fundamental geometric properties of the stack of semistable sheaves, with particular emphasis on local quotient presentations and slice models. 

In Section~\ref{sect:construction} the existence of a good moduli space is established by checking the conditions given in \cite[Theorem~1.2]{AlperFS}.  

In the final section, Section~\ref{sect:properties}, we identify the points of the moduli space as representing $S$-equivalence classes of sheaves and establish separatedness and properness of the moduli space, completing our investigation.  

\subsection*{Acknowledgements}
The authors would like to thank Jarod Alper and Jochen Heinloth for discussions regarding algebraic stacks and good moduli spaces. Moreover, they want to express their deep gratitude to Peter Heinzner who in the early stages of the project invited MT to Bochum twice and contributed to the discussions that lead to the development of the approach pursued here. 

DG was partially supported by the DFG-Collaborative Research Center
  SFB/TR 45 ``Periods, moduli spaces and arithmetic of algebraic varieties''. Moreover, he would like to thank the Institut \'Elie Cartan de Lorraine for hospitality during a visit to Nancy in September 2015.

\section{Basic notions and first properties}\label{sect:prelim}

\subsection{Global conventions}
We work over the field of complex numbers. All manifolds are assumed to be connected. We will work on a fixed complex projective manifold $X$ endowed with a cohomology class $\alpha \in H^{1,1}(X, \mathbb{R})$ that can be represented by a K\"ahler form $\omega$; i.e., $\alpha = [\omega]$.

\subsection{Sheaf extensions and automorphisms}\label{subsect:sheafextensions}
Here we recall a few facts about sheaf extensions and state two  lemmata to be used later in the paper.  
We start by considering extensions of $\cO_X$-modules  over a ringed space $(X,\cO_X)$, where $\cO_X$ is a sheaf of $\C$-algebras. It is known that the $\C$-vector space $\E(E_2,E_1)$ of classes of extensions of $E_2$ by $E_1$ modulo Yoneda equivalence is canonically isomorphic  to $\Ext^1_{\cO_X}(E_2,E_1)$, cf.~\cite[Exercise III.6.1]{Hartshorne}, \cite[Exercise A3.26]{Eisenbud_CommAlg}. 
Morphisms  $\alpha\in \Hom_{\cO_X}(E_1,E'_1)$, $\beta\in\Hom_{\cO_X}(E'_2,E_2)$ induce natural linear maps 
$\alpha_*:\E(E_2,E_1)\to\E(E_2,E'_1)$, 
$\beta^*:\E(E_2,E_1)\to\E(E'_2,E_1)$, cf.~\cite[Exercise A3.26]{Eisenbud_CommAlg}. On the $\Ext^1$-side  these correspond exactly to the linear maps induced by $\alpha$ and $\beta$ using the natural morphisms $\alpha_*:\Hom_{\cO_X}(E_2,E_1)\to\Hom_{\cO_X}(E_2,E'_1)$, $\beta^*:\Hom_{\cO_X}(E_2,E_1)\to\Hom_{\cO_X}(E'_2,E_1)$. It follows that $\alpha_*\circ\beta^*=\beta^*\circ\alpha_*$ in $\Hom_\C(\E(E_2,E_1),\E(E'_2,E'_1))$. 

\begin{rem}\label{remark:extensions}
The following particular cases of the above construction will be used in the sequel:
\begin{enumerate}
\item When $E'_1=E_1$, $E'_2=E_2$, we get a natural action of 
$\Aut(E_1)\times\Aut(E_2)$ on $\E(E_2,E_1)$ by 
$$(\alpha,\beta)(\xi):= (\alpha_*\circ(\beta^{-1})^*)(\xi)=((\beta^{-1})^*\circ\alpha_*)(\xi),$$
for $\alpha\in \Aut(E_1)$, $\beta\in\Aut(E_2)$, $\xi\in\E(E_2,E_1)$, cf. \cite[Chapter 7]{LePotier_Lectures}. If moreover $E_1=E_2=:E$, we get a natural  action of $\Aut(E)$ on $\E(E,E)$ by
$$\alpha(\xi):=(\alpha,\alpha)(\xi)= (\alpha_*\circ(\alpha^{-1})^*)(\xi)=((\alpha^{-1})^*\circ\alpha_*)(\xi).$$
\item By functoriality, when $j\in \Hom_{\cO_X}(E_1,E'_1)$ admits a retract or  when  
$p\in\Hom_{\cO_X}(E'_2,E_2)$ admits a section, we get  injective maps $j_*:\E(E_2,E_1)\to\E(E_2,E'_1)$, 
 $p^*:\E(E_2,E_1)\to\E(E'_2,E_1)$.
\end{enumerate}
\end{rem}

We next show how these considerations apply to  infinitesimal deformations of sheaves. For simplicity we restrict ourselves to the case when $X$ is a compact analytic space and the sheaves involved are coherent, but note that similar arguments work in the category of coherent sheaves over schemes. We denote by $\nearrow:=(point, \C[t])$ the double point, where $\C[t]:=\C[T]/(T^2)$ is the algebra of dual numbers over $\C$. Let $F$ be a coherent sheaf on $X$ and $(S,0)$ a germ of a complex space. A deformation of $F$ with base $S$ is a pair $(\cF,\phi)$ where $\cF$ is a coherent sheaf on $X\times S$ flat over $S$ and $\phi:\cF_0\to F$ is an isomorphism. Two deformations $(\cF,\phi)$, $(\cF',\phi')$ of $F$ with base $S$ are called isomorphic if there exists an isomorphism of sheaves $\Phi:\cF\to\cF'$ such that $\phi'\circ\Phi=\phi$,
 \cite[Section 4.2.2]{Palamodov-versal}. 
There is a natural bijection between the set of isomorphism classes of deformations of $F$ with base $\nearrow$ also called (first-order deformations) and the vector space $\E(F,F)$, \cite[Theorem 2.7]{Hartshorne_Deformations}. 
Any deformation of $F$ with base $S$ gives rise to a "tangent map" $T_0S\to\E(F,F)$. Finally we mention that the automorphism group of $F$ naturally acts on the set of (isomorphism classes of) deformations of $F$ with base $S$ by $g(\cF,\phi):=(\cF,g\circ\phi)$, for $g\in\Aut(F)$. 

\begin{lemma}\label{lemma: deformations}
The natural identification between the set of isomorphism classes of first-order deformations of $F$  and $\E(F,F)$ is $\Aut(F)$-equivariant. 
\end{lemma}

Fix now two coherent sheaves $E_1$, $E_2$ on $X$. In our set-up $W:=\E(E_2,E_1)$ is a finite dimensional complex vector space and there exists a universal extension
\begin{equation}
0\to E_{1,W}\to \cE\to E_{2,W}\to 0
\end{equation}
on $X\times W$, 
\cite[Chapter 7]{LePotier_Lectures}. The central fibre of the universal extension is a trivial extension
\begin{equation}
0\to E_{1}\overset{\alpha}{\longrightarrow} \cE_0\overset{\beta}{\longrightarrow} E_{2}\to 0
\end{equation}
on $X$.
Fixing a section $s\in \Hom_{\cO_X}(E_2,\cE_0)$ gives us an isomorphism $\phi:\cE_0\to E_1\oplus E_2$ hence a deformation $(\cE,\phi)$ of $E_1\oplus E_2$ with base $(W,0)$.

\begin{lemma}\label{lemma:extensions-deformations}
The tangent map $\E(E_2,E_1)\to \E( E_1\oplus E_2,  E_1\oplus E_2)$ to the deformation $(\cE,\phi)$ induced by the universal extension  coincides  with the natural inclusion $(\phi\circ \alpha)_*\circ (\beta\circ\phi^{-1})^*=(\beta\circ\phi^{-1})^*\circ(\phi\circ \alpha)_*$ given by Remark \ref{remark:extensions}(2) and is equivariant with respect to the group homomorphism $\Aut(E_1)\times\Aut(E_2)\to\Aut(E_1\oplus E_2)$ and the actions described in Remark \ref{remark:extensions}(1). 
\end{lemma}
\begin{proof}
We will check that the images in $\E( E_1\oplus E_2,  E_1\oplus E_2)$ of the class $\xi\in\E(E_2,E_1)$ of any extension
\begin{equation}
\label{equation:xi}
0\to E_{1}\overset{j}{\longrightarrow} E\overset{p}{\longrightarrow} E_{2}\to 0
\end{equation} 
of coherent sheaves on $X$  induced in the two different ways described in the statement coincide. The second part of the Lemma will follow from this.

Consider in addition a trivial extension 
\begin{equation}
\label{equation: trivial extension}
0\to E_{1}\overset{\alpha}{\longrightarrow} E_0\overset{\beta}{\longrightarrow} E_{2}\to 0
\end{equation}
and fix a section $s:E_2\to E_0$ of $\beta$ and the induced retraction $r:E_0\to E_1$ of $\alpha$. Then it is directly seen that the class $\alpha_*(\xi)\in \E(E_2,E_0)$ is represented by the second line of the following commutative diagram:
 \begin{equation*}
\begin{gathered}
\xymatrix{ 
0\ar[r] & E_1 \ar[d]^{\alpha}\ar[r]^{j} 
& E\ar[r]^p \ar[d]^{\left(\begin{smallmatrix} \id\\ 0
  \end{smallmatrix}\right)} & E_2 \ar[d]^{\id} \ar[r] & 0 \\
0\ar[r] & E_0 \ar[r]^{\left(\begin{smallmatrix} j\circ r\\ \beta
  \end{smallmatrix}\right)} & E\oplus E_2 \ar[r]^{\left(\begin{smallmatrix} p&0
  \end{smallmatrix}\right)}  & E_2 \ar[r] & 0
}
\end{gathered}
\end{equation*}
and that the first line of the diagram 
  \begin{equation*}
\begin{gathered}
\xymatrix{ 
0\ar[r] 
& E_0 \ar[d]^{\id}\ar[r]^{\left(\begin{smallmatrix} j\circ r\\ s\circ\beta\end{smallmatrix}\right)}
 & E\oplus E_0\ar[r]^{\left(\begin{smallmatrix} s\circ p&\alpha\circ r \end{smallmatrix}\right)}  \ar[d]^{\left(\begin{smallmatrix} \id& 0\\0&\beta
  \end{smallmatrix}\right)} & E_0 \ar[d]^{\beta} \ar[r] & 0 \\
0\ar[r] & E_0 \ar[r]^{\left(\begin{smallmatrix} j\circ r\\ \beta
  \end{smallmatrix}\right)} & E\oplus E_2 \ar[r]^{\left(\begin{smallmatrix} p&0
  \end{smallmatrix}\right)}  & E_2 \ar[r] & 0
}
\end{gathered}
\end{equation*}
 represents $\beta^*(\alpha_*(\xi))$. We will later use this first line under the form
   \begin{equation}
   \label{equation:beta(alpha(xi))}
\begin{gathered}
\xymatrix{ 
0\ar[r] 
& E_0 \ar[r]^{\gamma_1}
& E_0\oplus E\ar[r]^{\delta_1}
& E_0  \ar[r] & 0, 
}
\end{gathered}
\end{equation}
with $\gamma_1=\left(\begin{smallmatrix} s\circ\beta\\j\circ r\end{smallmatrix}\right)$ and $\delta_1=\left(\begin{smallmatrix} \alpha\circ r&s\circ p \end{smallmatrix}\right)$.

We next look at the restriction of the universal extension over the embedded double point $\nearrow$ at $0$ in $W$, which  points in the direction of $\xi$. We will write $2X:=X\times \nearrow\subset X\times W$,  $X:=X\times 0\subset X\times \nearrow\subset X\times W$ and denote by $\cO_X[t]:=\cO_X\otimes_\C\C[t]=\cO_{2X}$ the structure ring of $2X$ and by $\pi:2X\to X$ the projection. The class of this extension will be given by $t\pi^*(\xi)\in \Ext^1_{\cO_{2X}}(E_{2,2X},E_{1,2X})$, where $E_{i,2X}:=\pi^*E_i$. The multiplication by $t$ on $\Ext^1_{\cO_{2X}}(F_2,F_1)$ is given by $\mu_*:\Ext^1_{\cO_{2X}}(F_2,F_1)\to\Ext^1_{\cO_{2X}}(F_2,F_1)$, where $\mu=\mu_t:F_1\to F_1$ is the multiplication morphism by $t$ on $F_1$. We apply it to the element $\pi^*(\xi)$ which is represented by the pull-back of the extension \ref{equation:xi} to $2X$. We first note that the inverse image $\pi^*F=F\otimes_\C\C[t]$ to $2X$ through $\pi$ of a $\cO_X$-module $F$ is isomorphic as a $\cO_X$-module to $F\oplus F$. On $F\oplus F$ multiplication by $t$ is given by the $\cO_X$-linear operator 
$\left(\begin{smallmatrix}0& 0\\ \id&0\end{smallmatrix}\right)$ which gives  $F\oplus F$ its $\cO_X[t]$-module structure back. In these terms the extension \ref{equation:xi} pulls back to $2X$ to
   \begin{equation}
\begin{gathered}
\xymatrix{ 
0\ar[r] 
& E_1\oplus E_1 \ar[r]^{\left(\begin{smallmatrix} j&0\\0&j\end{smallmatrix}\right)}
 & E\oplus E\ar[r]^{\left(\begin{smallmatrix} p&0\\0&p \end{smallmatrix}\right)}   & E_2\oplus E_2  \ar[r] & 0. 
}
\end{gathered}
\end{equation}
and the lower line of the following diagram of $\cO_{2X}$-modules represents $t\pi^*(\xi)$:
 \begin{equation}
\label{equation: t xi}
\begin{gathered}
\xymatrix{ 
0\ar[r] 
& E_1\oplus E_1 
\ar[d]^{\left(\begin{smallmatrix} 0& 0\\ \id&0  \end{smallmatrix}\right)} \ar[r]^{\left(\begin{smallmatrix} j&0\\0&j\end{smallmatrix}\right)}
 & E\oplus E
 \ar[d]^>>>>>>>{\left(\begin{smallmatrix} 0& s\circ p\\ \id&0
  \end{smallmatrix}\right)} \ar[r]^{\left(\begin{smallmatrix} p&0\\0&p \end{smallmatrix}\right)}   
  & E_2\oplus E_2 \ar[d]^{\id} \ar[r] & 0. 
 \\
0\ar[r] 
& E_1\oplus E_1  
\ar[r]^{\left(\begin{smallmatrix} \alpha&0\\0&j
  \end{smallmatrix}\right)} 
& E_0\oplus E 
\ar[r]^{\left(\begin{smallmatrix} 0&p\\ \beta&0
  \end{smallmatrix}\right)}  
  & E_2\oplus E_2 \ar[r] & 0,
}
\end{gathered}
\end{equation}
where the $t$-multiplication on the term $E_0\oplus E $ is given by the operator $\left(\begin{smallmatrix} 0& s\circ p\\ j\circ r&0
  \end{smallmatrix}\right)$. We may write this line also under the form 
  $ 0\to \pi^*E_1\to\cE\to\pi^*E_2\to0$. Tensoring it by $0\to\cO_X\stackrel{t}\rightarrow\cO_{2X}\to\cO_X\to0$ leads to a commutative diagram 
   \begin{equation}
\label{equation: diagram 3x3}
\begin{gathered}
\xymatrix{ 
 &0\ar[d]&0\ar[d]&0\ar[d]&
\\
0\ar[r] 
& E_1 \ar[d]^{\alpha} \ar[r]
& \pi^*E_1\ar[d] \ar[r]
& E_1 \ar[d] \ar[r] & 0 
\\
0\ar[r] 
& E_0\ar[d]^\beta\ar[r]^{\gamma_2}
&\cE\ar[d]\ar[r]^{\delta_2}  
& E_0\ar[d]^\beta\ar[r] & 0
\\
0\ar[r] 
& E_2 \ar[d]^{\alpha} \ar[r]
& \pi^*E_2 \ar[r]\ar[d] 
& E_2 \ar[d] \ar[r] & 0 
\\
  &0&0&0&
}
\end{gathered}
\end{equation}
of $\cO_{2X}$-modules with exact rows and columns and the extension class of its middle row is the first order deformation induced by the family $\cE$. If we replace  now  this middle row by the sequence \ref{equation:beta(alpha(xi))} representing $\beta^*(\alpha_*(\xi))$ we get again a diagram of $\cO_{2X}$-modules with exact rows and columns
\begin{equation}
\label{equation: diagram 3x3 new}
\begin{gathered}
\xymatrix{ 
 &0\ar[d]&0\ar[d]&0\ar[d]&
\\
0\ar[r] 
& E_1 \ar[d]^{\alpha} 
\ar[r]^{\left(\begin{smallmatrix}  0\\ \id
  \end{smallmatrix}\right)}
& E_1\oplus E_1
\ar[d]^{\left(\begin{smallmatrix} \alpha&0\\0&j
  \end{smallmatrix}\right)} 
  \ar[r]^{\left(\begin{smallmatrix} \id& 0
  \end{smallmatrix}\right)}
& E_1 \ar[d]^{\alpha} \ar[r] & 0 
\\
0\ar[r] 
& E_0\ar[d]^\beta\ar[r]^{\gamma_1}
&E_0\oplus E
\ar[d]^{\left(\begin{smallmatrix} 0&p\\ \beta&0
  \end{smallmatrix}\right)} 
\ar[r]^{\delta_1}  
& E_0\ar[d]^\beta\ar[r] & 0
\\
0\ar[r] 
& E_2 
\ar[d] 
\ar[r]^{\left(\begin{smallmatrix}  0\\ \id
  \end{smallmatrix}\right)}
& E_2\oplus E_2 
\ar[r]^{\left(\begin{smallmatrix} \id& 0
  \end{smallmatrix}\right)}
\ar[d] 
& E_2 \ar[d] \ar[r] & 0 
\\
  &0&0&0&
}
\end{gathered}
\end{equation} 
and it follows that the differences $\delta_2-\delta_1$ and $\gamma_2-\gamma_1$ factorize through morphisms of $\cO_{2X}$-modules $\left(\begin{smallmatrix}  u&0
  \end{smallmatrix}\right):E_2\oplus E_2\to E_1$ and $\left(\begin{smallmatrix}  0\\ v
  \end{smallmatrix}\right):E_2\to E_1\oplus E_1$, respectively, with $u,v\in\Hom_{\cO_X}(E_2,E_1)$. Putting $\epsilon=\left(\begin{smallmatrix} 0&\alpha\circ u\circ p\\ 0&0
  \end{smallmatrix}\right)$ 
  and 
  $\epsilon'=\left(\begin{smallmatrix} 0&0\\ j\circ v\circ\beta&0
  \end{smallmatrix}\right)$ in $\End_{\cO_X}(E_0\oplus E)$
  we get $\epsilon^2=(\epsilon')^2=\epsilon\circ\epsilon'=
  \epsilon'\circ\epsilon=0$, $\delta_2-\delta_1=\delta_1\circ\epsilon$, $\gamma_2-\gamma_1=\epsilon'\circ\gamma_1$, $\epsilon\circ\gamma_1=0$ and $\delta_1\circ\epsilon'=0$. Then 
  the commutative diagram 
   \begin{equation*}
\begin{gathered}
\xymatrix{ 
0\ar[r] & E_0 \ar[d]^{\id}\ar[r]^{\gamma_2} 
& E_0\oplus E\ar[r]^{\delta_2} \ar[d]^{\id+\epsilon-\epsilon'} & E_0 \ar[d]^{\id} \ar[r] & 0 \\
0\ar[r] & E_0 \ar[r]^{\gamma_1} 
& E_0\oplus E\ar[r]^{\delta_1} & E_0  \ar[r] & 0 
}
\end{gathered}
\end{equation*}
  shows that the desired extensions lie in the same class in $\Ext^1_{\cO_X}(E_0,E_0)$ and we are done.  
\end{proof}

\subsection{Semistable coherent sheaves}\label{subsect:coherentsheaves}

We will work on a fixed complex projective manifold $X$ endowed with a cohomology class $\omega \in H^{1,1}(X, \mathbb{R})$ that can be represented by a K\"ahler form. This class will serve as a polarisation which will help us to define  Gieseker-Maruyama-semistability for coherent sheaves on $X$, cf. \cite[Definition 11.1]{GRT1}. We start by studying basic properties of semistable sheaves. For simplicity, we will only consider the case of torsion-free sheaves, although most properties are valid for pure coherent sheaves. Later on, we will focus on the case of rank two torsion-free sheaves.

\begin{defi}
 Let $E$ be a coherent sheaf on $X$. Its \emph{Hilbert-polynomial (with respect to $\omega$)} is defined as the polynomial function (with coefficients in $\mathbb{C}$) that is given by
 $$P_E(m):=P^{\omega}_E(m) := \int_X \ch(E)\, e^{m\omega}\, \Todd(X),$$
 where $\ch(E)$ and $\Todd(X)$ denote the Chern character of $E$ and the Todd class of $X$, respectively.
 If $E$ is torsion-free and non-zero we define its \emph{reduced Hilbert-polynomial} as
 $$p_E:=p^{\omega}_E := \frac{P_E}{\rank E}. $$
 We will say that $E$ is \emph{(Gieseker-Maruyama-)stable (with respect to $\omega$)} and \emph{semistable}, respectively, if $E$ is torsion-free and if for any coherent subsheaf $0\neq F\subsetneq E$ one has $p_F<p_E$ and $p_F\le p_E$, respectively. We will call $E$ \emph{polystable} if it splits as a direct sum of stable subsheaves having the same reduced Hilbert-polynomial.
 If $E$ is semistable but not stable we will say that it is \emph{properly semistable}.
\end{defi}

The usual relations to slope-stability (with respect to $\omega$), which will also be referred to as $\mu$-stability, continue to hold in this context, namely:\\

\centerline{
$\mu$-stable $\Rightarrow$ stable $\Rightarrow$ semistable
$\Rightarrow$ $\mu$-semistable,}

\noindent\ \\ 
cf.~\cite[Lem.~1.6.4]{HL}.
In particular, the boundedness result \cite[Proposition 6.3]{GrebToma} for $\mu$-semistable sheaves implies:

\begin{Prop}[Boundedness]\label{boundedness}
Let $X$ be a $d$-dimensional projective manifold and let $K$ a compact subset of the K\"ahler cone $\mathcal{K}(X) \subset H^{1,1}(X, \mathbb{R})$ of $X$. Fix a natural number $r>0$ and classes $c_i\in H^{2i}(X,\R)$, $i = 1, \dots, d$. Then, the family
of rank $r$ torsion-free sheaves $E$ with $c_i(E)=c_i$ that are semistable with respect to some polarisation contained in $K$ is bounded.
\end{Prop}  

The proofs of the following three basic results are standard and therefore left to the reader, cf.~\cite[Prop.~1.2.7]{HL}, \cite[Prop.~3.1]{Seshadri}\footnote{Seshadri formulates and proves the corresponding result for slope-semistable vector bundles of degree zero over a Riemann surface; Gieseker-Maruyama-semistability is the correct higher-dimensional semistability condition to make this work in general.} and \cite[Section 9.3]{LePotier_Lectures}, and finally \cite[Prop.~1.5.2]{HL}, respectively. 

\begin{lemma}\label{lem:morphisms}
Let $E$ and $E'$ be  semistable sheaves on $X$  and let $\phi :E\to E' $ be a non-zero morphism of $\cO_X$-modules. Then $p_E\le p_{E'}$. If equality holds, then $\im(\phi)$ is semistable and $p_{\im(\phi)}=p_E= p_{E'}$. If moreover the rank of $\im(\phi)$ coincides with the rank of $E$ or with the rank of $E'$ then $\im(\phi)$ is isomorphic to $E$ or to $E'$ respectively.
\end{lemma}

\begin{prop}
\label{prop:abelian_category}
The full subcategory $\Coh^{ss}{(X,\omega,p)}$ of the category of coherent sheaves on $X$, whose objects are the semistable sheaves with fixed reduced Hilbert polynomial $p$ and the zero-sheaf, is abelian, noetherian and artinian. 
\end{prop}

\begin{prop}[Jordan-H\"older filtrations]\label{prop:JH}
 Any semistable sheaf on $X$ admits a Jordan-H\"older filtration  in the sense of \cite[Def.~1.5.1]{HL} (with respect to $\omega$-stability). The associated graded sheaf is unique up to isomorphism.
\end{prop}

The derivation of the following result is less formal and requires deeper insight into the geometry of Douady spaces.

\begin{thm}[Openness of (semi)stability]\label{thm:openness-semistability}
 Let $(S,0)$ be a complex space germ and $\cE$  be a coherent sheaf on $X\times S$ that is flat over $S$. If the fibre of $\cE$ over $0\in S$ is (semi)stable, then the fibres of $\cE$ over any point in a neighbourhood of $0$ in $S$ are likewise (semi)stable. 
\end{thm}

\begin{proof}
The proof of \cite[Corollary 5.3]{TomaLimitareaI} immediately adapts to our situation to show that the relative Douady space $D_S(\cE)_{\le b}$  of quotients of $\cE$ with
degrees bounded from above by $b$ is proper over $S$; details will appear in \cite{TomaLimitareaII}. Using this as well as \cite[Lemma 4.3]{TomaLimitareaI} to replace  Grothendieck's Lemma, we may then prove openness of (semi)stability as in the classical case of ample polarisations, as presented for example in~\cite[Proposition 2.3.1]{HL}. 
\end{proof}

\subsection{Semistable sheaves of rank two}\label{subsect:rank2}

The next result gives a classification of semistable sheaves of rank two on a fixed projective manifold $X$ that is endowed with a given K\"ahler form $\omega$ and computes the automorphism group for all the resulting classes.

\begin{prop}[Classification of semistable sheaves]\label{prop:ss sheaves}
Any semistable sheaf $E$ of rank $2$ on $X$ falls into exactly one of the following classes:  
\begin{enumerate}[leftmargin=*, font=\bfseries ]
\item[\emph{(1)}] {\bf Polystable sheaves}
\begin{enumerate}
\item[\emph{(a)}] {\em Stable sheaves.} In this case $\Aut(E)\cong\C^*$. 
\item[\emph{(b)}] {\em Decomposable sheaves of the form $L_1\oplus L_2$ with $L_1\ncong L_2$ and $P_{L_1}=P_{L_2}$}. In this case $\Aut(E)\cong\C^*\times \C^*$, and $\Hom(L_1,E)$,  $\Hom(E, L_1)$, $\Hom(E,L_2)$, as well as $\Hom(L_2,E)$ are one-dimensional.
\item[\emph{(c)}] {\em Decomposable sheaves of the form $L\oplus L$}. In this case $\Aut(E)\cong\GL(2,\C)$ and $\Hom(L,E)$,  $\Hom(E, L)$ are two-dimensional.
\end{enumerate}
\item[\emph{(2)}] {\bf Non-polystable sheaves}
\begin{enumerate}
\item[\emph{(a)}] {\em Centres of non-trivial extensions of the form $0\to L_1\to E\to L_2\to 0$ with $L_1\ncong L_2$ and $P_{L_1}=P_{L_2}$}. In this case, we have $\Aut(E)\cong\C^*$, $\Hom(L_1,E)\cong\C$,  $\Hom(E, L_1)=0$, $\Hom(E,L_2)\cong\C$, and  $\Hom(L_2,E)=0$.
\item[\emph{(b)}] {\em Centres of non-trivial extensions of the form $0\to L\stackrel{\alpha}\rightarrow E\stackrel{\beta}\rightarrow L\to0$.} In this case $\Aut(E)=\{ a\cdot \Id_E+c\cdot\alpha\circ\beta \ | \ a\in\C^*, \ c\in\C\}\cong\C^*\times\C$, $\Hom(L,E)\cong\C$ and $\Hom(E,L)\cong\C$.
\end{enumerate}
\end{enumerate}
In all cases listed above, $L_1$, $L_2$, $L$ are torsion-free sheaves of rank one on $X$. 
\end{prop}
\begin{proof}
The classification follows easily from the existence and uniqueness of Jordan-H\"older filtrations, see Proposition~\ref{prop:JH}. We will hence only compute the automorphism groups and the homomorphism groups here, relying mostly on Lemma~\ref{lem:morphisms}. The three cases listed under (1) are clear. To deal with the cases listed under (2), let $E$ be the centre of a non-trivial extension of the form
$$0\to L_1\stackrel{\alpha}\longrightarrow E\stackrel{\beta}\longrightarrow L_2\to0$$
with $P_{L_1}=P_{L_2}$. 

In case $L_1\ncong L_2$, using the fact that the extension is assumed to be non-split we immediately get $\Hom(L_1,E)\cong\C$,  $\Hom(E, L_1)=0$, $\Hom(E,L_2)\cong\C$, and $\Hom(L_2,E)=0$. Applying now $\Hom(E,\cdot)$ to the defining exact sequence of $E$ we obtain $\Hom(E,E)\cong\C$, hence $\Aut(E)\cong\C^*$. 

Suppose now that $L_1\cong L_2=:L$. Let $\sigma\in\Hom(E,L)$. Then, $\sigma\circ\alpha=0$, otherwise $\sigma$ would be a retraction of $\alpha$, contradicting the assumption that the extension is non-split. Consequently, $\sigma$ factorizes through $\beta$, i.e. $\sigma=c\beta$ for some $c\in\C$. In particular, $\Hom(E, L)\cong\C$. Similarly we get $\Hom(L,E)\cong\C$.  Applying as before $\Hom(E,\cdot)$ to the defining exact sequence of $E$, we get
$$0\to\Hom(E,L)\stackrel{\alpha\circ \cdot}\longrightarrow\Hom(E, E)\stackrel{\beta\circ\cdot}\longrightarrow\Hom(E, L).$$
The image of an element $\phi$ in  $\Hom(E, E)$ through 
the map $\Hom(E, E)\stackrel{\beta\circ\cdot}\longrightarrow\Hom(E, L)$ is of the form $a\beta$ for some 
$a\in\C$, with $a\neq0$ if $\phi\in\Aut(E)$. 
With this notation $\beta\circ(\phi-a\Id_E)=0$, hence $\phi-a\Id_E=\alpha\circ c\beta=c\cdot \alpha\circ\beta$ and the desired description of $\Aut(E)$ follows.
\end{proof}

\begin{cor}\label{cor:extensions}
 Up to a multiplicative constant every properly semistable sheaf $E$ of rank $2$ on $X$ gives rise to a unique extension 
$$ 0\to L_1\to E\to L_2\to 0,$$
with rank one torsion free sheaves  $L_1$, $L_2$ on $X$ such that $P_{L_1}=P_{L_2}$.
\end{cor}

\subsection{Basic geometric properties of the stack of semistable sheaves} \label{subsect:quotient_presentations} 

We consider the stack $\cX:=\cC oh_{(X,\omega),\tau}^{ss}$ of semistable sheaves  on $(X,\omega)$ with fixed rank and Chern classes; the latter data will be collected in a vector $\tau=(r, c_1, \dots, c_{2\dim X})$, which we call the \emph{type} of the sheaves under consideration. This is an algebraic stack locally of finite type over $\C$ since it satisfies Artin's axioms \cite[Theorem 2.20]{Alper}; see also \cite[Theorem 2.19]{AlperHR}. See also the beginning of Section~\ref{sect:construction} for more information on the basic properties of $\cX$.

\subsubsection{Quotient stack realisation}\label{subsubsect:quotient_realisation}

The stack $\cX$ may be realized as a quotient stack in the sense of \cite[Definition 3.1]{Alper} in the usual way; we quickly recall the construction, which is explained for example in  \cite[Sect.~4.3]{HL}: choose an ample line bundle $\cO_X(1)$ and an integer $m$ such that all semistable sheaves (with respect to $\omega$) with fixed rank and Chern classes $\tau$ on $X$ are $m$-regular with respect to $\cO_X(1)$. 
This is possible since we have boundedness for such sheaves by Proposition \ref{boundedness}. Since the rank and the Chern classes of the sheaves $F$ under consideration are fixed, by $m$-regularity and Riemann-Roch we obtain that $h^0(F(m))$ is constant, equal to $N \in \mathbb{N}$. 
Setting $V:=\C^N$, $\cH:=V\otimes_\C\cO_X(-m)$, we obtain for any $F$ as above an epimorphism of $\cO_X$-modules 
$\rho:\cH\to F$
as soon as we have fixed an isomorphism $V\to H^0(F(m))$.
Moreover, the induced map $H^0(\rho(m)):H^0(\cH(m))\to H^0(F(m))$ is bijective. We thus get a point $[\rho:\cH\to F]$ in the open (quasi-projective) subscheme $R$ of $\Quot_\cH$ of semistable quotients $F$ of $\cH$ with type $\tau$ that induce isomorphisms at the level of $H^0(\rho(m)):H^0(\cH(m))\to H^0(F(m))$. The natural action of the linear group $G:=\GL(V)$ on $V$ induces an action on $\Quot_\cH$ leaving the open subset $R$ invariant. Let $\cF$ be the universal quotient sheaf restricted to $X\times R$. It is a $G$-sheaf and it allows to define an isomorphism from the quotient stack $[R/G]$ to $\cX$. Indeed, an object of  $[R/G]$ 
is a triple $(T,\pi:P\to T, f:P\to R)$, where $T$ is a scheme, $\pi$ is a principal $G$-bundle and $f$ is a $G$-equivariant morphism. Then the $G$-sheaf obtained from $\cF$ by pullback to $X\times P$ gives a flat family of semistable sheaves on $X$ parametrised by $T$ and thus an object of $\cX$. 
Conversely if $\cE$ is a flat family of semistable sheaves of type $\tau$ on $X$ parametrised by a scheme $S$, then  as in the proof of \cite[Lemma 4.3.1]{HL} the frame bundle $R(\cE(m))$ associated to it gives an object $(S,R(\cE(m))\to S,R(\cE(m))\to R)$ of $[R/G]$. 

In the subsequent discussion, we will use the following notation: If $G$ is an algebraic group and $X$ is a $G$-scheme, then for $x \in X(\mathbb{C})$ we denote by $[x]_G$ the image of $x$ under the morphism $X \to [X/G]$. We will also use the same notion for the associated points in the corresponding topological spaces $|X|$ and $|[X / G]|$.

\subsubsection{Closed points and closures of points} We will characterise closed points in terms of polystability and show that polystable degenerations are unique. Grauert semicontinuity, see \cite[Prop.~III.12.8]{Hartshorne},  is the key principle at work here. With a view towards the discussion carried out in subsequent parts of the paper we will restrict ourselves to the case of coherent sheaves having rank two.
\begin{prop}[Characterising closed points]\label{prop:characterising_closed_points}
Let $z \in R$ be a closed point. Then, the following are equivalent.
\begin{enumerate}
 \item The point $[z]_G \in  |[R/G]| \cong |\cX|$ is closed.
 \item The $G$-orbit $G\acts z \subset R$ is closed.
 \item The sheaf $\cF_z$ is polystable.
\end{enumerate}
\end{prop}
\begin{proof}
The equivalence ``(1) $\Leftrightarrow$ (2)'' follows directly from the definitions. In order to show ``(3) $\Rightarrow$ (2)/(1)'', assume that $\cF_z$ is properly semistable.  Then, $\cF_z$ can be realised as a non-trivial extension $0 \to L_1 \to \cF_z \to L_2 \to 0$, see Proposition~\ref{prop:ss sheaves}. Consequently, $\cF_z$ degenerates to $L_1 \oplus L_2$ over the affine line, and therefore does not give a closed point of $\cX$. It remains to show that orbits of polystable sheaves are closed. This however follows from a Grauert semicontinuity argument completely analogous to \cite[Lem.~4.7]{Gieseker}. 
\end{proof}
The following now is a consequence of Proposition~\ref{prop:ss sheaves}.
\begin{cor}\label{cor:reductive_stabiliser}
 Every closed point of $\cX$ has linearly reductive stabiliser. 
\end{cor}
Next, we look at closures of non-closed points.
\begin{prop}[Uniqueness of polystable degenerations]
 For any $\C$-point $y$ of $\cX$, there exists a unique closed point in $\overline{\{y\}} \subset |\cX|$.
\end{prop}
\begin{proof}
 If $\cF_z$ is polystable, by Proposition~\ref{prop:characterising_closed_points} the corresponding point $y = [z]_G$ is closed, so there is nothing to show.
 
 If $\cF_z$ is properly semistable, then clearly the closed point corresponding to the polystable sheaf $\mathrm{gr}^{JH}(\cF_z)$ lies in $\overline{\{[z]_G\}}$. Suppose that there is another closed point $x$ of $|\cX|$ lying in $\overline{\{[z]_G\}}$, and let $\cE$ be a polystable sheaf representing $x$. Grauert semicontinuity then implies that 
 \begin{align*}
  \dim_\C \Hom(\mathrm{gr}^{JH}(\cF_z), \cE) &\geq  \dim_\C\Hom(\mathrm{gr}^{JH}(\cF_z), \cF_z) \quad \text{ and}\\
  \dim_\C \Hom(\cE, \mathrm{gr}^{JH}(\cF_z)) &\geq \dim_\C\Hom(\cF_z, \mathrm{gr}^{JH}(\cF_z)),
 \end{align*}
 from which we quickly deduce that the polystable sheaf $\cE$ has to be isomorphic to $\mathrm{gr}^{JH}(\cF_z)$.
\end{proof}

\subsubsection{Slices and local quotient presentations}
We note that by construction $R$ admits $G$-equivariant locally closed embeddings into the projective spaces associated with finite-dimensional complex $G$-representations, arising from  natural $G$-linearised ample line bundles on the Quot-scheme induced by $\cO_X(1)$, see \cite[p.~101]{HL}. This fact will be used in the proof of the subsequent result, which provides rather explicit local quotient presentations for the stack $\cX$. We continue to use the notation established in Section~\ref{subsubsect:quotient_realisation}.

\begin{prop}[Local quotient presentation induced by slice]\label{prop:slice}
 Let $E$ be a semi\-stable sheaf on $X$ corresponding to a closed point $x \in \cX (\mathbb{C})$.  Let $s \in R$ project to the closed point $[s]_{G} \in [R/G]$ that is mapped to $x$ by the isomorphism $[R/G] \cong \cX$ established above. Then, there exists a $G_s$-invariant, locally closed, affine subscheme $S$ in $R$ with $s \in S$ such that $T_sR = T_sS \oplus T_s(G\cdot s)$, such that the morphism $G \times S \to R$ is smooth, and such that the induced morphism $f: [S/G_s] \to \cX$ is \'etale and affine, maps the point $0 :=[s]_{G_s} \in [S/G_s] (\mathbb{C})$ to $x$, and induces an isomorphism of stabiliser groups $G_s = \mathrm{Aut}_{[S/G_s]}(0) \overset{\cong}{\longrightarrow}\mathrm{Aut}_{\cX}(x) \cong \mathrm{Aut}(E)$; i.e., $f$ is a local quotient presentation of $\cX$ at $x$ in the sense of \cite[Def.~2.1]{AlperFS}.
\end{prop}
\begin{proof}
 By Corollary~\ref{cor:reductive_stabiliser}, the stabiliser subgroup $\Aut_{\cX}(x) \cong \Aut(E)$ is linearly reductive. Consequently, the proof of the claim presented in Remark 3.7 and Lemma 3.6 of \cite{AlperKresch} continues to work even without the normality assumption made there, if we replace the application of Sumihiro's Theorem (which uses the normality assumption) by the observation made in the paragraph preceding the proposition that in our setup right from the start $R$ comes equipped with a $G$-equivariant locally closed embedding into the projective space associated with a finite-dimensional complex $G$-representation. Alternatively, see \cite[Props.~9.6 and 9.7]{JoyceSong}.
\end{proof}

\begin{cor}[Slice is stabiliser-preserving]\label{cor:stab_preserving}
 In the setup of Proposition~\ref{prop:slice}, let $t \in S \subset R$ and let $H_t$ be the stabiliser group of the action of $H := G_s$ on $S$ at the point $t\in S$. Then, we have $H_t = G_t$.
 As a consequence, we obtain
 \[H_t \cong \mathrm{Aut}(\mathcal{F}_t)\]
 under the morphism of stabiliser groups induced by $f: [S/G_s] \to \cX$. 
\end{cor}
\begin{proof}
 As the $H$-action on $S \subset R$ is obtained by restricting the $G$-action to the subgroup $H$, we clearly have the inclusion
 \begin{equation}\label{eq:one_inclusion}
  H_t \subset G_t \quad \quad \forall t \in S.
 \end{equation}
Moreover, as for all $t \in S$ the stabiliser subgroup $G_t$ is isomorphic to the automorphism group $\mathrm{Aut}(\mathcal{F}_t)$ of the corresponding member of the family $\mathcal{F}$, and is therefore connected by Proposition~\ref{prop:ss sheaves}, it suffices to show that the two groups appearing in \eqref{eq:one_inclusion} have the same dimension. Now, the fact that $f: [S/G_s] \to \cX$ is \'etale implies that the natural $G$-equivariant morphism 
\[\varphi: G \times_H S \to R, [g, t]\; \mapsto g \acts t\]
from the twisted product $G\times_H S := (G \times S)/H$\footnote{The quotient is taken with respect to the proper action $h\acts (g, t) := (gh^{-1}, h\acts t)$.} to $R$ is \'etale. In particular, the restriction of $\varphi$ to any $G$-orbit is \'etale. We conclude that 
\[\dim H_t = \dim G_{[e,t]} = \dim G_{\varphi([e,t])} = \dim G_t,\]
as desired. This concludes the proof.
\end{proof}

  \begin{prop}[Slice provides semi-universal deformation]
   In the situation of Propositions~\ref{prop:slice}, the analytic germ $(S^{an}, s)$ of $S^{an}$ at $s$ together with the restriction $(\cU^{an}, s) := (\cF|_{(S,s) \times X})^{an}$ of the universal family $\cF$ of $R$ to $(S^{an}, s)$ is a semi-universal deformation of $E$.
  \end{prop}
\begin{proof}
As both $[S/G_s]$ and $\cX$ are algebraic stacks, there exist formal miniversal deformations $\widehat{\Def}(x)$ and $\widehat{\Def}([s])$ of $x\in \cX(\C)$ and $[s] \in [\Spec A/G_x](\C)$. Moreover, the local quotient presentation establishes an isomorphism of formal schemes $\hat{f}:\widehat{\Def}([s]) \to \widehat{\Def}(x)$. We will check that the first space is isomorphic to the formal completion $\widehat S$ of $S$ at $s$.

We claim that the natural morphism $(S, s)\to ([S/G_s], 0)$ is formally versal at $s$ in the sense of \cite[Def.~A.8]{AlperHR}. For this, we check the assumptions of \cite[Prop.~A.9]{AlperHR}: Both $s$ and $0$ are closed points. Moreover, the morphism $S \to [S/G_s]$ is representable and smooth. Hence, the induced map of the $0$-th infinitesimal neighbourhoods $S^{[0]} \to [S/G_s]^{[0]}$ is likewise representable, and for every $n \in \mathbb{N}$ the induced map of $n$-th infinitesimal neighbourhoods $S^{[n]} \to [S/G_s]^{[n]}$ is smooth. Finally, the stabiliser of $[S/G_s]$ at $s$, which is equal to $G_s \cong \mathrm{Aut}(E)$, is reductive. Consequently, part (2) of \cite[Prop.~A.9]{AlperHR} implies that $(S, s)\to ([S/G_s], 0)$ is formally versal at $s$, as claimed. Moreover, as $s$ is a $G_s$-fixed point, the induced map on tangent spaces $T_s S \to T_0 [S/G_s]$ is an isomorphism. 

As a consequence, we see that the restriction $\widehat{\cU}$ of the universal family $\cF$ to the formal completion $\widehat S$ of $S$ at $s$ is an object of $\cX$ over $\widehat S$ that is formally miniversal at $s$ in the sense of \cite[Def.~2.8]{Alper}. Moreover, $\cU^{an} = (\cF|_{S \times X})^{an}$ obviously provides an analytification of $\widehat{\cU}$. It follows from the fact that a versal deformation of $E$ exists and from \cite[Satz 8.2]{FlennerHabil} that the germ $(S,s)$ of $S$ at $s$ together with the restriction of $\cU^{an}$ to this germ is a semi-universal deformation of $E$. 
\end{proof}

\begin{rem}
 Using analytic stacks, an alternative proof can be given as follows: As in the above proof, one easily checks that the map $(S,s)^{an} \to [S/G_s]^{an} = [S^{an}/G_s]$ is smooth and the induced map on tangent spaces is an isomorphism. These two conditions are equivalent to the conditions in the definition of a semi-universal family, cf.~\cite[p.~19]{KosarewStieber90}.
\end{rem}

For later usage, we note two properties of the $\mathrm{Aut}(E)$-action on its semi-universal space $(S, s)^{an}$: 

\begin{lemma}[Action of the homothety subgroup]\label{lemma:action_of_homotheties}
 In the situation of Proposition~\ref{prop:slice}, the subgroup of homotheties $\mathbb{C}^*\cdot \Id_E$ of $E$  acts trivially on the semi-universal deformation space $S$. 
\end{lemma}

\begin{proof}
 Under the identification of $\mathrm{Aut}(E)$ with $G_s \subset \GL(V)$, the subgroup $\mathbb{C}^*\cdot \Id_E$ is mapped to $\mathbb{C}^*\cdot \Id_V$, which acts trivially on $\mathrm{Quot}_\cH$, see \cite[proof of Lem.~4.3.2]{HL}.
\end{proof}

\begin{lemma}[Action on tangent space]\label{lem:action_on_tangent}
  Using the identification of $\mathrm{Aut}(E)$ with $G_s$, the tangent space of $(S, s)$ is $\Aut(E)$-equivariantly isomorphic to $\E (E,E)$, where the action on the latter space is as described in Section~\ref{subsect:sheafextensions}.
\end{lemma}

\section{Construction of the moduli space}\label{sect:construction}

 The aim of this section is to construct a good  moduli space for  the stack $\cX:=\cC oh_{(X,\omega),\tau}^{ss}$ of semistable sheaves with fixed type  $\tau = (2, c_1, \dots, c_{2 \dim X})$ on $(X,\omega)$. As we have seen in Section~\ref{subsect:quotient_presentations}, whose notation we will use in our subsequent arguments, this is an algebraic stack locally of finite type over $\C$, which can be realized as a quotient stack $\cX \cong [R/G]$. Using this global quotient presentation as well as the local slice models also constructed in Section~\ref{subsect:quotient_presentations}, we will prove that the algebraic stack $\cX$ admits a good moduli space by checking the conditions of \cite[Theorem 1.2]{AlperFS}. 
 
To collect some general information about $\cX$, we start by noting that $\cX$ has affine diagonal. For this we use the chart $R\to\cX$. Thus to check that the map $\cX\to\cX\times_{\Spec \C}\cX$ is affine comes to showing that $R\times G\to R\times R$, $(q,g)\mapsto(q,qg)$ is affine. But this is a map of $R$-schemes, where $R\times G$ is affine over $R$ and $R\times R$ is separated over $R$, hence the conclusion follows by \cite[Lemma 28.11.11(2)]{stacks-project}.
In a similar way one checks that $\cX$ is quasi-separated; one  looks again at the map $R\times G\to R\times R$, which is quasi-compact.

The following is the main result of this section. 
\begin{thm}[Existence]\label{thm:existence} The stack $\cC oh_{(X,\omega),\tau}^{ss}$  of $\omega$-semistable sheaves with rank two and fixed Chern classes admits a good moduli space $M^{ss}:=M_{(X,\omega),\tau}^{ss}$.
\end{thm}

\begin{proof} Recall the criteria for the existence of a good moduli space given in \cite[Theorem 1.2]{AlperFS}:
\begin{enumerate}
\item For any $\C$-point $y\in\cX(\C)$, the closed substack $\overline{\{y\}} \subset \cX$ admits a good moduli space.
\item For any closed point $x\in\cX(\C)$ there exists a local quotient presentation $f:\cW\to\cX$ around $x$ in the sense of \cite[Definition 2.1]{AlperFS} such that
\begin{enumerate}
\item $f$ sends closed points to closed points,
\item $f$ is stabilizer preserving at closed points of $\cW$.
\end{enumerate}
\end{enumerate}

Following the structure of these conditions, our proof is subdivided into two big steps, establishing Condition (1) and (2), respectively. 
\subsection{Condition (1)} If $y$ is closed, the condition is easily verified, as the stabiliser group of $y$ is affine. So, suppose that $y$ corresponds to a semistable sheaf $F$ appearing as an extension 
$$0\to L_1\to F\to L_2\to 0,$$ 
where $L_1$, $L_2$ are rank one sheaves with the same  Hilbert-polynomial $P$.

To deal with such extensions we consider the stack of flags $\cD rap(X;P,P)$ whose objects over $S$ are sheaves $\cF_1\subset \cF_2$ on $X\times S$ such that $\cF_1$ and $ \cF_2/\cF_1$  are flat over $S$ with fixed Hilbert-polynomials $P$ and $P$ respectively,  cf.~\cite[Sect.~2.A.1]{HL}. In fact since the Hilbert polynomial $P^\omega$ is constant in flat families, $\cD rap(X; P,P)$ is a closed and open substack of the stack $\Quot((X\times\cC oh(X))/\cC oh(X),\cF)$, where $\cF$ is the universal sheaf on $X\times\cC oh(X)$\footnote{See \cite{HallRydh_Hilbert} for a general result on $\Quot((X\times\cC oh(X))/\cC oh(X),\cF)$, but note that in our case $X\times\cC oh(X)\to\cC oh(X)$ is locally of finite presentation.}.

We claim that the forgetful morphism $\phi:\cD rap(X;P,P)\to\cC oh(X)$, $(\cF_1,\cF_2)\mapsto\cF_2$ is proper in the sense of \cite[D\'efinition 3.10.1]{LaumonMB}. Indeed, for any object of $\cC oh(X)$, given by a flat family $\cF$ of coherent sheaves on $X$ parametrised by a scheme $S$ we get a Cartesian diagram
 \begin{equation*}
\label{diagr: Drap}
\begin{gathered}
\xymatrix{ 
S\times_{\cC oh(X)}{\cD rap(X; P,P)}\ar[r] \ar[d]^{\phi_S} & \cD rap(X;P, P) \ar[d]^\phi  \\
S \ar[r] & \cC oh(X)
}
\end{gathered}
\end{equation*}
in which the first vertical map comes from the universal family of quotients of $\cF$ relative to $S$ of Hilbert polynomial $P$. Thus, the morphism $\phi_S$ is the natural map $\Quot_{\cF/S}(P)\to S$, which is proper by \cite{TomaLimitareaII}; see \cite[Corollary 5.3]{TomaLimitareaI} for a proof of the absolute case of this claim. Hence, also $\phi$ is  proper. In our case since $X$ is projective it follows in fact that $\phi$ is even projective, but we will not need this fact.

Let now $P=\frac{1}{2}P_E^\omega$, where $E$ is a coherent sheaf on $X$ of type $\tau$. We  consider the substack $\cD rap^{ss}(X; P,P)\subset\cD rap(X;P)$ of sheaves $\cF_1\subset \cF_2$ as before such that moreover $\cF_2$ is semistable of Hilbert polynomial $2P$. 
Note that in this situation the quotient $\cF_2/\cF_1$ has rank one and no torsion, since otherwise the saturation of $\cF_1$ in $\cF_2$ would contradict the semistability of $\cF_2$. 
We thus get a morphism $\cD rap^{ss}(X; P,P)\to\cM^s(X,P)\times\cM^s(X,P)$, $(\cF_1,\cF_2)\mapsto(\cF_1,\cF_2/\cF_1)$ whose fibre over a closed point $(L_1,L_2)\in\cM^s(X,P)\times\cM^s(X,P)$  is the closed substack \[\cD rap(X; L_1,L_2)\subset\cD rap^{ss}(X; P,P)\] of flags  $\cF_1\subset \cF_2$ such that the fibres of $\cF_1$ are isomorphic to $L_1$ and the  fibres of $\cF_2/\cF_1$ are isomorphic to $L_2$. Here $\cM^s(X,P)$ denotes the moduli space of rank one stable sheaves on $X$ with Hilbert polynomial $P$.

\subsubsection{The case when $L_1\cong L_2$.} We look at points $y\in \cC oh(X)(\C)$ corresponding to coherent sheaves $F$ on $X$ that sit in a short exact sequence of the type
$$0\to L\to F\to L\to0.$$
Such coherent sheaves $F$ are semistable with respect to any polarisation on $X$, so also with respect to the ample line bundle $\cO_X(1)$. Let $h=c_1(\cO_X(1))$ and consider the open substack $\cU:=\cC oh_{(X,\omega),\tau}^{ss}\cap\cC oh_{(X,h),\tau}^{ss}$ of $\cC oh_{(X,\omega),\tau}^{ss}$. Then,
$y\in\cU(\C)$.  
The stack  $\cD rap(X; L,L)$ is proper over both $\cC oh_{(X,\omega),\tau}^{ss}$ and $\cC oh_{(X,h),\tau}^{ss}$, and its image contains the respective closures of ${y}$ in $\cC oh_{(X,\omega),\tau}^{ss}$ and $\cC oh_{(X,h),\tau}^{ss}$, which therefore coincide. It is thus enough to show that the closure $\overline{\{y\}}$ in $\cC oh_{(X,h),\tau}^{ss}$ admits a good moduli space. 
But the latter stack admits a good moduli space itself \cite[Example 8.7]{Alper13}, so the same will hold for the closed substack $\overline{\{y\}}$ by \cite[Lemma 4.14]{Alper13}.

\subsubsection{The case when $L_1\ncong L_2$.} 
As $L_1\ncong L_2$ the morphism $\cD rap(X; L_1,L_2)\to \cC oh^{ss}(X,2P)$ is proper with finite fibres, hence finite, cf.~\cite[Lemma 36.38.4]{stacks-project}, and therefore affine. Let $\cY_{L_1,L_2}\subset \cC oh^{ss}(X,2P)$ be its image. We will use \cite[Proposition 1.4]{AlperFS} to show that $\cY_{L_1,L_2}$ admits a good moduli space. From this and \cite[Lemma 4.14]{Alper13}  it will follow that the closed substack $\overline{\{y\}}\subset \cY_{L_1,L_2}$ likewise admits a good moduli space.

Set $W:=\Ext^1(L_2,L_1)$. In Lemma \ref{lemma: extensions} we will show that $\cD rap(X; L_1,L_2)\cong[W/\C^*\times\C^*]$, so that in particular $\cD rap(X; L_1,L_2)$ admits a separated good moduli space. 
It remains to check that $\cY_{L_1,L_2}$ is a global quotient stack and  admits local quotient presentations.  
As a closed substack of the global quotient stack $\cX\cong[R/G]$, $\cY_{L_1,L_2}$ is a global quotient stack as well, cf.~\cite[Definition 3.4]{Alper}. 
Moreover,the morphism $\cD rap(X; L_1,L_2)\to\cY_{L_1,L_2}$ is itself a local quotient presentation for $\cY_{L_1,L_2}$, \cite[Definition 2.1]{AlperFS}. Indeed, the only condition of that definition which we haven't yet checked is the fact that the morphism $\cD rap(X; L_1,L_2)\to\cY_{L_1,L_2}$ is \'etale. By \cite[Lemma 40.17.3]{stacks-project} it is enough to show that the morphism $\cD rap(X; L_1,L_2)\to \cC oh^{ss}(X,2P)$ is unramified, since finiteness has already been established. 
This follows from the differential study of the Quot scheme,  \cite[Proposition 2.2.7]{HL}, applied to diagrams of the form
\begin{equation}
\label{diagr: Drap(L_1,L_2)}
\begin{gathered}
\xymatrix{ 
S\times_{\cC oh^{ss}(X,2P)}{\cD rap(X; L_1,L_2)}\ar[r] \ar[d] & \cD rap(X; L_1,L_2) \ar[d] \\
S \ar[r] & \cC oh^{ss}(X,2P)
}
\end{gathered}
\end{equation}
as before and from the fact that $\Hom(L_1,L_2)=0$, as $L_1\ncong L_2$.

\subsection{Condition (2)}

We next turn our attention to condition (2). Let $x\in\cX(\C)$ be a closed point and $G_x$ its stabilizer. We will do a case by case analysis depending on the type of a representative $E$ of $x$.

\subsubsection{The stable case.}
The case when $E$ is stable is quickly dealt with. By openness of stability, see Theorem~\ref{thm:openness-semistability}, it suffices to construct a local quotient presentation at $[E]$ with the desired properties in the open substack $\cC oh_{(X,\omega),\tau}^{s} \subset \cX$ of stable sheaves on $X$. We consider the corresponding open $G$-invariant subspace $R^s \subset R$ of stable quotients and choose a point $p \in R^s$ mapping to $x$ under the natural map $[R/G] \to \cX$.  We note as a first point that every $G$-orbit in $R^s$ is closed in $R$ by Proposition~\ref{prop:characterising_closed_points}, as a second point that for every point $p \in R^s$ the stabiliser group $G_p$ is isomorphic to $\mathbb{C}^*$ as $\cE_p$ is simple, and as a third point that it follows from Lemma~\ref{lemma:action_of_homotheties} that $G_p$ acts trivially on the slice $S \ni p$ whose existence is guaranteed by Proposition~\ref{prop:slice} and which, shrinking $S$ if necessary, we may assume to be contained in $R^s$. As every $G$-orbit in $R^s$ is closed in $R$, condition 2(a) is fulfilled for the quotient presentation induced by $S$, whereas condition 2(b) is guaranteed to hold by Corollary~\ref{cor:stab_preserving}. This concludes the discussion of the stable case. 

\subsubsection{The case of a polystable point with $\Aut(E)\cong\GL(2,\C)$.}
If $E$ is polystable with $\Aut(E)\cong\GL(2,\C)$, $x$ is a point in the open substack $\cC oh_{(X,\omega),\tau}^{ss}\cap\cC oh_{(X,h),\tau}^{ss}$ of $\cC oh_{(X,\omega),\tau}^{ss}$. Let $R^{h\text{-ss}}$ be the corresponding $G$-invariant open subscheme of $R$ consisting of $h$-semistable quotients and let $R^{ss}$ be the $G$-invariant open subscheme of $R$ consisting of $\omega$-semistable quotients. Both subschemes contain the $G$-orbit $G\acts p$  corresponding to $E$. As $\cC oh_{(X,h),\tau}^{ss}$ admits a good moduli space $M^{h\text{-ss}}$ and as $x$ is a closed point of $\cC oh_{(X,h),\tau}^{ss}$, there exists an open, $G$-invariant subscheme $\cU \subset R^{h\text{-ss}} \cap R^{ss}$ that contains $x$ and is saturated with respect to the moduli map $R^{h\text{-ss}} \to M^{h\text{-ss}}$. It follows that the restriction of the moduli map to $\cU$ yields a good quotient $\cU \to \cU\hq G \hookrightarrow M^{h\text{-ss}}$ for the $G$-action on $\cU$. The desired quotient presentation is then produced by an application of Luna's slice theorem, see for example \cite{D_Luna}, 
at the closed orbit $G\acts p \subset \cU$. 

\subsubsection{The case of a polystable point with $\Aut(E)\cong\C^*\times\C^*$.}
Under the assumption, the point $x$ is in the image $\cY$ of $\cD rap^{ss}(X;P,P)$ which is proper over $\cX$. Moreover, inside $\cD rap^{ss}(X;P,P)$ we have a closed substack corresponding to the condition $\cF_1\cong\cF_2/\cF_1$. Let $\cZ$ be the image of this closed substack in $\cY$. The point   $x$ lies in the complement of $\cZ$, so we may assume that the image of the local quotient presentation $f$ guaranteed by Proposition~\ref{prop:slice} is contained in the complement of $\cZ$ too. We will use the notation of that Proposition throughout the rest of the proof.

Recall the following properties of the action of $\Aut(E)$ on $S = \mathrm{Spec}(A)$:
\begin{enumerate}
 \item The subgroup of homotheties $\C^*\cdot \Id_E \subset \mathrm{Aut}(E)$ acts trivially on $S$, see Lemma~\ref{lemma:action_of_homotheties}; the action of $\Aut(E)$ hence factors over an action of $\mathbb{C}^* \cong (\mathbb{C}^* \times \mathbb{C}^*)/ \mathbb{C}^*$ on $S$.
 \item The fibres $F$ over the fixed points for the action of $\Aut(E)$ have $\Aut(F)\cong\C^*\times\C^*$ and conversely, if $\mathrm{Aut}(\mathcal{F}_t) = \C^* \times \C^*$ for some $t \in S$, then $t$ is contained in the $\mathrm{Aut}(E)$-fixed point set in $S$, see Corollary~\ref{cor:stab_preserving}.
 \end{enumerate}

 Finally, we claim that, possibly after $S$ has been shrinked further, the orbits of non-polystable fibres $F$ by the $\Aut(E)$-action are not closed. 

Suppose by contradiction that the assertion does not hold. By semicontinuity arguments the set of points of $S$ parametrising non-polystable fibres is constructible, and so is the set of points belonging to closed orbits, as follows from Luna's slice theorem. Both sets are invariant under the $\Aut(E)$-action. If the closure of their intersection does not contain $s$, we just shrink $S$ so that it no longer intersects this closure. If it does contain $s$, we get a curve $C$ through the image $o$ of $s$ in $S \hq \Aut(E) := \mathrm{Spec}(A^{\Aut(E)})$ whose general points correspond to non-trivial closed $\Aut(E)$-orbits in $S$ parametrising non-polystable sheaves. 
 
Let $Y\subset \Spec A$ be an irreducible component of the inverse image of $C$ in $\Spec A$ containing such a general orbit. 
Then $Y$ is a surface with a $\C^*$-action and a good quotient $\pi :Y \twoheadrightarrow C \subset S \hq \Aut(E)$ .
The only point of $Y$ corresponding to a polystable sheaf is $s$. Let $Y^\circ:=Y\setminus\{ s\}$ and let $q:X\times Y\to Y$, and $\pi:Y\to C$ be the natural  projections.
For simplicity we denote again by $\cF$ the universal sheaf on  $X\times Y$. 

Any fibre $F$ of $\cF$ over a point $y$ of $Y$ appears as the middle term of an extension $0\to L_1\to F\to L_2\to 0$,
where $L_1$ and $L_2$ are stable of fixed Hilbert polynomial $P$ and non-isomorphic. 
If $s \neq y \in Y$, this extension is non-trivial and unique up to a multiplicative constant in $\C^*$, see Corollary \ref{cor:extensions}. 
Moreover, the sheaves parametrised by points lying in a fibre of $\pi$ over some point of $C\setminus \{o\}$ are all isomorphic, and thus they all correspond to the same extension, again up to multiplication by a non-zero constant.
Let $\cF_s=L_{1,s}\oplus L_{2,s}$. The natural morphism $\Quot_{\cF/Y}(P)\to Y$ is one-to-one over $Y^\circ$, whereas its fibre over $s$ has two points corresponding to $L_{1,s}$ and $ L_{2,s}$.

These properties imply that only one of the two quotients $L_{1,s}$ and $ L_{2,s}$ over $s$, say $ L_{2,s}$, is in the closure $Y'$ of $\Quot_{\cF/Y}(P)_{Y^\circ}$ in $\Quot_{\cF/Y}(P)$.  
We denote by $s'$ the point of $Y'$ lying over $Y$ and by $Y'_o$ the fibre over $o$ of the composition $Y'\to Y\to C$.
On $X\times Y'$ we thus get two rank one universal sheaves, the universal kernel and the universal quotient, which we denote by $\cL_1$ and $\cL_2$. 
The universal extension $0\to \cL_1\to \cF\to \cL_2\to 0$
on $X\times Y'$ restricted to $X\times Y'_o$  gives rise to a natural map $Y'_o\to W$, where $W:=\E( L_{2,s},L_{1,s})=\Ext^1_{\cO_X}(L_{2,s},L_{1,s})$ is the space of extensions of $L_{2,s}$ by $L_{1,s}$, see \cite[Corollary 3.4]{Lange}.  At the level of germs of complex analytic spaces, we get the following commutative diagram:
\begin{equation}
\label{diagr:germs}
\begin{gathered}
\xymatrix{
(Y',s')\ar[d] & (Y'_o,s')\ar[d]\ar[r] \ar[l] & (W,0)\ar[d] \\
(Y,s) \ar[dr]& (Y_o, s) \ar[l]\ar[d] \ar[r] & (S,s)\ar[d]\\
&(C,o)\ar[r]&(S \hq \Aut(E), o).
}
\end{gathered}
\end{equation}

The automorphism group $\Aut(E) \cong \mathbb{C}^* \times \mathbb{C}^*$ acts equivariantly on the induced diagram of the respective tangent spaces of the germs above.  
The map\footnote{Cf.~Lemma~\ref{lem:action_on_tangent}.} \[T_0W\to T_s S \cong\E(L_{1,s}\oplus L_{2,s},L_{1,s}\oplus L_{2,s}),\] is the one described by Lemma \ref{lemma:extensions-deformations}, and the group  action is induced by the action of $\Aut(L_{1,s})\times\Aut(L_{2,s})\cong\C^*\times\C^*$; in particular, note that the diagonal $\C^*\subset\C^*\times\C^*$ operates trivially on both sides, cf.~the proof of Lemma \ref{lemma: extensions}. Almost by definition, the character of the induced $\C^*$-action on $W$ is $+1$ or $-1$ depending on the isomorphism $\C^*\cong(\C^*\times\C^*)/\C^*$ we have chosen. In the sequel we will assume it to be 
 $+1$.  Choose a $\mathbb{C}^*$-equivariant closed embedding $\psi :S \hookrightarrow V$ into a finite-dimensional $\mathbb{C}^*$ representation space $V$. By composing with the translation by the $\mathbb{C}^*$-fixed point $\psi(s)$ if necessary, we may assume that $\psi(s) = 0 \in V$. Let $V=V_+\oplus V_0\oplus V_-$ be the decomposition of $V$ into subspaces according to the sign of the characters of the $\C^*$-action on $V$.
From Diagram~\eqref{diagr:germs} and the consideration regarding the weight of the action on $W$ we infer that $\psi$ embeds $Y_o$ into $V_+$. 
We claim that this implies that $\psi$ embeds the reduced space $Y_{\mathrm{red}}$ into $V_+\oplus V_0$. Indeed, if not, there would exist a sequence of points $z_n=(z_{n,+},z_{n,0},z_{n,-})$ in $\psi(Y_{\mathrm{red}})\setminus (V_+\oplus V_0)$ converging to $0$ in $V$. 
Then, we can find a sequence of elements $\lambda_n\in \C^*$ with $\lim_{n \to \infty} \lambda_n = 0$ such that $\Vert\lambda_n \acts z_{n,-} \Vert=1$ for each $n\in \N$. It follows that $(\lambda_n \acts (z_{n,+},z_{n,0}))_{n\in \N}$ converges to $0$ in $V_+\oplus V_0$ and that a subsequence of $(\lambda_n \acts z_{n})_{n \in \N}$ converges to a point of norm $1$ in $\psi(Y_{\mathrm{red}})\cap V_-$. Such a point would lie in $Y_o$ contradicting the fact that $Y_o$ is mapped to $V_+$ by $\psi$. Thus, $Y_{\mathrm{red}}$ is embedded into $V_+\oplus V_0$, as claimed\footnote{An alternative proof can be given using \cite[Rem.~1.1 and Lem.~1.7]{MR1273375} after normalising $Y_{\mathrm{red}}$. }.
  However, the fact that the only closed orbits of the $\C^*$-action on $V_+\oplus V_0$ are the fixed points contradicts our assumptions on $Y$.
  
This concludes the proof of Theorem~\ref{thm:existence}, establishing the existence of a good moduli space $M_{(X,\omega),\tau}^{ss}$. 
\end{proof}

\begin{lemma}\label{lemma: extensions}
If $L_1\ncong L_2$, then 
$\cD rap(X; L_1,L_2)\cong[W/\C^*\times\C^*]$. 
\end{lemma}
\begin{proof} We only sketch the line of the argument, which is most likely already present somewhere in the literature on the subject.
 
We choose to identify $\C^*\times\C^*$ with $\Aut(L_1)\times \Aut(L_2)$. Then  the induced action on  
$W=\Ext^1(X; L_2,L_1)$ is given by $w(\theta_1,\theta_2)=\theta_1\theta_2^{-1}w$. 
With this convention if $w\in W$ is the class of an extension
$0\to L_1\stackrel{\alpha}\longrightarrow F\stackrel{\beta}\longrightarrow L_2\to 0$ then $\theta_1\theta_2^{-1}w$ is represented by the second line of the diagram
\begin{equation*}
\label{diagr: extensions}
\begin{gathered}
\xymatrix{ 
0\ar[r]& 
L_1\ar[r]^\alpha \ar[d]^{\theta_1\Id_{L_1}} & 
F\ar[r]^\beta \ar[d]^{\Id_{F}} & 
L_2\ar[r] \ar[d]^{\theta_2\Id_{L_2}}& 
0 \\
0\ar[r]& L_1\ar[r]^{\theta_1^{-1}\alpha} & F\ar[r]^{\theta_2\beta} & L_2\ar[r]& 0. 
}
\end{gathered}
\end{equation*}

An object of  $[W/\C^*\times\C^*]$ is a triple $(T, P\stackrel{\pi}\longrightarrow T, P\stackrel{f}\longrightarrow W)$, where T is a scheme, $ P\stackrel{\pi}\longrightarrow T$ is a principal $\C^*\times\C^*$-bundle and $P\stackrel{f}\longrightarrow W$ is a  $\C^*\times\C^*$-equivariant morphism, cf. \cite[Definition 3.1]{Alper}. To such an object we associate an object of $\cD rap(X; L_1,L_2)$ in the following way. By \cite[Chapter 7]{LePotier_Lectures}, 
\cite[Corollary 3.4]{Lange} there exists a universal extension 
\begin{equation}
\label{diagr: universal}
0\to L_{1,W}\to \cF\to L_{2,W}\to 0
\end{equation}
on $W\times X$ which we pull back to $P\times X$. The action of $\C^*\times\C^*$ on $W$ induces a $\C^*\times\C^*$-linearisation on the sheaves $L_{1,P}$ and $\cF_P$, which thus descend to  $T\times X$ and give the desired object in $\cD rap(X; L_1,L_2)$ over $T$, \cite[Theorem 4.2.14]{HL}.

For the converse we use the fact that the moduli space $M^s(X,P)$ admits local universal families (cf. \cite[Appendix 4D. VI]{HL}), so for any object $(\cF_1,\cF_2)$ of $\cD rap(X; L_1,L_2)$ over $S$, one has $\cF_1\cong L_{1,S}\otimes\cL_{1,X}$ and $\cF_2/\cF_1\cong L_{2,S}\otimes\cL_{2,X}$, for suitable line bundles $\cL_{1}, \cL_2$ on $S$.
Let $P_1\to S$, $P_2\to S$ be the $\C^*$-principal bundle associated with the line bundles $\cL_{1}, \cL_2$ on $S$ and let $P:=P_1\times_SP_2$. On $P\times X$ we get a "tautological" extension
$$0\to L_{1,P}\to \cF_2\to L_{2,P}\to 0,$$
which is the pullback of the universal extension \ref{diagr: universal} by means of some  equivariant morphism $P\stackrel{f}\longrightarrow W$ by \cite[Corollary 3.4]{Lange} again. The triple $(S,P\to S, P\stackrel{f}\longrightarrow W)$ is the corresponding object of  $[W/\C^*\times\C^*]$ that we were looking for.
\end{proof}

\section{Properties of the moduli space}\label{sect:properties}
We start by discussing functorial properties of $M_{(X,\omega),\tau}^{ss}$. Analogously to \cite[Section 4.1]{HL} we consider the functors $\underline{M'}:=\underline{M'}_{(X,\omega),\tau}:(Sch/\C)\to (Sets)$, $\underline{M}:=\underline{M'}/\sim$, where, for a scheme $S$ over $\C$, $\underline{M'}(S)$ is the set of isomorphism classes of flat families over $S$ of semistable sheaves of type $\tau$ on $X$ and two such families $F,E\in\underline{M'}(S)$ are equivalent through $\sim$
if there exists a line bundle $\cL\in\Pic(S)$ such that $E$ is isomorphic to $F\otimes \cL_X$. As explained in \cite[Section 4.1]{HL}, if an algebraic space $M$ corepresents the functor $\underline{M'}$, then it also corepresents $\underline{M}$ and the other way round. Finally, using \cite[Lemma 4.3.1]{HL} and \cite[Theorem 4.16(vi)]{Alper13} we get:

\begin{Prop}\label{prop:corepresentaion} $M_{(X,\omega),\tau}^{ss}$ corepresents the functors $\underline{M'}_{(X,\omega),\tau}$ and $\underline{M}_{(X,\omega),\tau}$.
\end{Prop}

Recall that two semistable sheaves on $X$ are called $S$-equivalent if their Jordan-H\"older graduations are isomorphic. By  \cite[Theorem 4.16(iv)]{Alper13}, \cite[Lemma 4.1.2]{HL} and our previous considerations we immediately obtain:

\begin{Prop}\label{prop:points} The closed points of $M_{(X,\omega),\tau}^{ss}$ correspond precisely to $S$-equivalence classes of semistable sheaves of type $\tau$ on $X$.
\end{Prop}

We next prove that the constructed moduli space is separated. This will follow from a refinement of Langton's  valuative criterion for separation which is only formulated in \cite{Langton} for two semistable sheaves, at least one of which is stable. But whereas Langton's theorem is stated for slope-semistable sheaves, we are working within the abelian category of Gieseker-Maruyama-semistable sheaves of fixed Hilbert polynomial with respect to $\omega$ (and $0$), cf.~Proposition \ref{prop:abelian_category}. This fact is essential for the following criterion to hold.  

The set-up is the following. We consider a discrete valuation ring $A$ over $\C$ with maximal ideal $\mg$ generated by a uniformising parameter $\pi$. We set $K$  the field of fractions of $A$.    We denote by $X_K:=X\times\Spec(K)$ the generic fibre and by $X_\C:=X\times\Spec(\C)$ the special fibre of $X_K:=X\times\Spec(K)$ and by $i:X_K\to X_A$, $j:X_\C:=X\to X_A$ the inclusion morphisms. We denote furthermore by $\xi$ and by $\Xi$ the generic points of $X_\C$ and of $X_K$, respectively. Note that $\cO_{X_A,\xi}$ is a discrete valuation ring with maximal ideal generated by $\pi$.

\begin{Prop}[Valuative criterion for separation]\label{prop:separation} Let $E_K$ be a 
 torsion-free sheaf of rank $r$ on $X_K$ and let $E_1, E_2\subset i_*E_K$ be two coherent sheaves on $X_A$ such that $i^*E_1=i^*E_2=E_K$ and such that $E_{1,\C}:=j^*E_1$ and $E_{1,\C}:=j^*E_2$ are semistable on $X_\C$. Then 
$$\mathrm{gr}^{JH}(j^*E_1)\cong\mathrm{gr}^{JH}(j^*E_2).$$
\end{Prop}
\begin{proof}
We follow closely Langton's proof, \cite[pages 101-102]{Langton}. 
Note that $E_1$ and $E_2$ are flat over $\Spec(A)$ since they are torsion free and $A$ is a discrete valuation ring. Moreover, since they coincide over $\Spec(K)$, their restrictions to $X_\C$ have the same Hilbert polynomial.

By \cite[Proposition 6]{Langton} any rank $r$ free $\cO_{X_A,\xi}$ submodule $M$ of $(E_K)_\Xi$ gives rise to a unique torsion-free coherent sheaf $E$ of $i_*E_K$ on $X_A$ such that  $i^*E=E_K$, $E_\xi=M$ and $j^*E$ is torsion-free on $X_\C$. 
Langton introduces an equivalence relation on such submodules by putting $M\sim\pi^nM$ and calls two equivalence classes $[M]$, $[M']$ {\it adjacent} if there exists a direct sum decomposition $M=N\oplus P$ such that $M'=N+\pi M$. 
Equivalent modules induce isomorphic extensions of $E_K$ to coherent subsheaves on $X_A$ as in  \cite[Proposition 6]{Langton}. 
This is no longer true in general for adjacent classes. In fact for adjacent classes $[M]$ and $[M']$ we may suppose that $M$ has a basis $(e_1,...,e_r)$ over $\cO_{X_A,\xi}$  such that for a suitable $s\in\{1,...,r\}$ the module $M'$ admits $(e_1, ..., e_s,\pi e_{s+1},...,\pi e_r)$ as a basis. 
If $E$, $E'$ denote the coherent sheaf extensions of $M$ and $M'$ to $X_A$, then the inclusion of $\cO_{X_A,\xi}$-modules $M'\subset M$ induces an inclusion of coherent sheaves $E'\subset E$ on $X_A$ which restricts to  a morphism $\alpha: E'_\C\to E_\C$ on $X_\C$, whose image is the unique saturated coherent subsheaf $F$ of $E_\C$ such that $F_\xi$ is the $\cO_{X_\C,\xi}$-vector space generated by $\bar e_1,...,\bar e_s$, where the elements $\bar e_j$ are the images of $e_j$ under $M\to (M\otimes \cO_{X_A,\xi}/\pi \cO_{X_A,\xi})$. 
In   \cite[Proposition 7]{Langton} it is shown that $F$ is saturated in $E_\C$ and that $E'$ appears as what is called an elementary transformation of $E$ and in particular that $\Ker(\alpha)\cong Coker(\alpha)$. 
One gets exact sequences 
$$0\to \Ker(\alpha)\to E'_\C\to \im(\alpha)\to0,$$
$$0\to\im(\alpha)\to E_\C\to\coker(\alpha)\to0$$
of torsion-free sheaves on $X_\C$.
Langton calls the passage from $[M]$ to $[M']$ an {\it edge} and denotes it by $[M] \ - \ [M']$.

{\it Remark. 
For an edge as above $E_\C$ and $E'_\C$ have the same  Hilbert polynomial  and in case $E_\C$ and  $\im(\alpha)$ are semistable with reduced Hilbert polynomial $p$ then also $\Ker(\alpha)$ and $E'_\C$ will be  semistable with reduced Hilbert polynomial $p$ by Proposition \ref{prop:abelian_category} and $\mathrm{gr}^{JH}(E_\C)\cong\mathrm{gr}^{JH}(E'_\C)$.}

We consider now the $\cO_{X_A,\xi}$-modules $E_{1,\xi}$, $E_{2,\xi}$. 
We can find a basis $(e_1,...,e_r)$ of  $E_{1,\xi}$ over $\cO_{X_A,\xi}$ such that   $(\pi^{m_1}e_1, ...,\pi^{m_r}e_r)$ is a basis of $E_{2,\xi}$, where $m_1$, ..., $m_r$ are suitable integers. 
Up to replacing $E_2$ by $\pi^nE_2$ for some $n$ and up to permuting the $e_i$-s  we may suppose that $m_1=0$ and that $m_1\le m_2\le..\le m_r$. 
We  now construct a sequence of $m_r$ edges $[M] \ - \ [M']$, $[M'] \ - \ [M'']$, ..., $[M^{(m_{r}-1)}] \ - \ [M^{(m_{r})}]$ as above starting at $[M]=[E_{1,\xi}]$ and ending at $ [M^{(m_{r})}]=[E_{2,\xi}]$ in the following way. 
If $s$ is such that $m_s=m_1=0$ and $m_{s+1}>m_s$, we set $M':=(e_1,...,e_s,\pi e_{s+1},...,\pi e_r)$ to be the $\cO_{X_A,\xi}$-submodule of $(E_K)_\Xi$ generated by the elements $e_1,...,e_s,\pi e_{s+1},...,\pi e_r$. 
We next set $M'':=(e_1,...,e_s,\pi^2 e_{s+1},...,\pi^2 e_r)$, ..., $M^{(m_{s+1})}:=(e_1,...,e_s,\pi^{m_{s+1}} e_{s+1},...,\pi^{m_{s+1}} e_r)$ and continue with  $$M^{(m_{s+1}+1)}:=(e_1,...,e_s,\pi^{m_{s+1}} e_{s+1},...,\pi^{m_{s+1}} e_t, \pi^{m_{s+1}+1} e_{t+1},...,\pi^{m_{s+1}+1} e_r),$$ where $t$ is such that $m_t=m_{s+1}$, $m_{t+1}>m_t$ and so on until we reach $E_{2,\xi}$. 
For this sequence of edges we denote by $E=E_1$, $E'$, ...,$E^{(m_r)}=E_2$ the associated sheaf extensions to $X_A$ and by  $\alpha_1, ..., \alpha_{m_r}$ the induced morphisms. 
By construction we have $\im(\alpha_1)=\im(\alpha_1\circ\alpha_2\circ...\circ\alpha_{m_r})$. 
By Proposition \ref{prop:abelian_category} it follows that $\im(\alpha_1)$ is semistable with reduced Hilbert polynomial $p:=p_{E_{1,\C}}=p_{ E_{2,\C}}$ so by the above Remark also $E'_\C$ is semistable with reduced Hilbert polynomial $p$ and $\mathrm{gr}^{JH}(E_\C)\cong\mathrm{gr}^{JH}(E'_\C)$. 
Iterating  this piece of argument we thus  obtain 
\[\mathrm{gr}^{JH}(E_{1,\C})\cong\mathrm{gr}^{JH}(E'_\C)\cong...\cong\mathrm{gr}^{JH}(E^{(m_r-1)}_\C)\cong\mathrm{gr}^{JH}(E_{2,\C}). \qedhere\]
\end{proof}

\begin{rem}\label{rem:proof-of-separation}
The proof of Proposition \ref{prop:separation} works for slope semistability and for more general semistability notions as  in \cite[Definition 2.2]{TomaCriteria} if one replaces the Jordan-H\"older graduation in the category $\Coh(X)$ by the Jordan-H\"older graduation in an appropriate quotient category $\Coh_{d,d'}(X)$ as defined in \cite{HL}. The problem is that for such semistability notions  Jordan-H\"older graduations  exist in  $\Coh(X)$, but are no longer unique, \cite[Proposition 2.1]{BuTeTo1bis}.
\end{rem}

\begin{cor}[Separation]\label{cor:separation}
 The moduli space $M_{(X,\omega),\tau}^{ss}$
is separated.
\end{cor}

\begin{Prop}[Properness]\label{prop:properness} The moduli space $M_{(X,\omega),\tau}^{ss}$
is proper.
\end{Prop}
\begin{proof}
This follows from the analogon of Langton's valuative criterion for properness proved in \cite{TomaCriteria}. 
\end{proof}

\def\cprime{$'$} \def\polhk#1{\setbox0=\hbox{#1}{\ooalign{\hidewidth
  \lower1.5ex\hbox{`}\hidewidth\crcr\unhbox0}}}
  \def\polhk#1{\setbox0=\hbox{#1}{\ooalign{\hidewidth
  \lower1.5ex\hbox{`}\hidewidth\crcr\unhbox0}}}
\providecommand{\bysame}{\leavevmode\hbox to3em{\hrulefill}\thinspace}
\providecommand{\MR}{\relax\ifhmode\unskip\space\fi MR }
\providecommand{\MRhref}[2]{%
  \href{http://www.ams.org/mathscinet-getitem?mr=#1}{#2}
}
\providecommand{\href}[2]{#2}

\medskip
\medskip
\begin{center}
\rule{0.4\textwidth}{0.4pt}
\end{center}
\medskip
\medskip
\end{document}